\documentclass{article}[12pt]
\usepackage{amsmath}
\usepackage{amsfonts}
\usepackage{amssymb}
\usepackage{graphicx}
\usepackage{epstopdf}
\usepackage{color,graphicx}

\providecommand{\U}[1]{\protect\rule{.1in}{.1in}}
\newtheorem{theorem}{Theorem}

\newtheorem{lemma}[theorem]{Lemma}

\newtheorem{proposition}[theorem]{Proposition}
\newtheorem{remark}[theorem]{Remark}

\newenvironment{proof}[1][Proof]{\noindent\textbf{#1.} }{\ \rule{0.5em}{0.5em}}

\title
 {Asymptotic expansions of the Helmholtz equation solutions using approximations of the Dirichlet to Neumann operator}

\author{
Souaad Lazergui\thanks{University of Mostaghanem, Department of Pure and Applied Mathematics, B.P. 188, 27000, Algeria. lazergui.souad@gmail.com}, 
Yassine Boubendir\thanks{New Jersey Institute of Technology, Department of Mathematical Sciences, University Heights, Newark NJ 07102, USA. boubendi@njit.edu}
}
\begin{document}
\maketitle
\begin{abstract}
This paper is concerned with the asymptotic expansions of the amplitude of the solution of the Helmholtz equation.  The original expansions were obtained using a pseudo-differential 
decomposition of the Dirichlet to Neumann operator. This work uses first and second order approximations of this operator to derive new asymptotic expressions of the normal derivative of the total field. The resulting expansions can be used to appropriately choose the  ansatz in the design of  high-frequency numerical solvers, such as those based on  integral equations, in order to produce more accurate approximation of the solutions around the shadow and deep shadow regions than the ones based on  
    the usual ansatz.
      
\end{abstract}

\section{Introduction} 


Studying the Helmholtz equation at the high-frequency regime is fundamental in both the 
theoretical understanding of the corresponding solutions and the derivation of appropriate
numerical schemes. The well-know asymptotic expansions  developed by Melrose and Taylor
\cite{tay1} have significantly contributed in this matter and were the key in the design
of several high-frequency integral equation methods. Indeed,  
integral equation methods are very efficient and widely used in the solution of acoustic 
scattering problems  (see e.g. \cite{Colton-Kress:83,Chandler-WildeEtAl12} and 
the references therein).  
However, the resulting linear systems are dense, ill-conditioned and with large size in particular
when the frequency increases. Several effective strategies have been proposed to overcome these 
difficulties \cite{Colton-Kress:83,Chandler-WildeEtAl12, Amini1,BrackhageWerner,Burton,Antoine-Bendali-Darbas:05,AntoineX,Antoine,
Levadoux,Panich,br-turc,turc9,Grenn1,Rokhlin:90,TongChew10, BrunoSayas,Yimaturc,5b,1b}. 
Despite this significant progress, integral formulations are limited at higher frequencies since the
numerical resolution of field oscillations can easily lead to impractical computational times. This
is why hybrid numerical methods based on a combination of
integral equations and asymptotic methods have found an increasing interest for the solution
of high-frequency scattering problems. Indeed, the methodologies developed in this connection
that specifically concern scattering from a smooth convex obstacle were first introduced 
 in \cite{AbboudEtAl94,AbboudEtAl95}. Several other works followed these    
\cite{oscar, fatih1, fatih2, Fernando1,victor,HuybrechsVandewalle07,Giladi07,EcevitOzen14,fatihnew} 
and mainly consist of improving and analyzing this kind of numerical algorithms 
in single and multiple scattering configurations.  
All these methods are mainly based on construction
of an appropriate ansatz for the solution of integral equations in the form of a highly oscillatory
function of known phase modulated by an
unknown slowly varying envelope, which is expected to generate linear systems quasi-independent of the frequency.

\begin{figure}
\centering
\includegraphics[height=60mm]{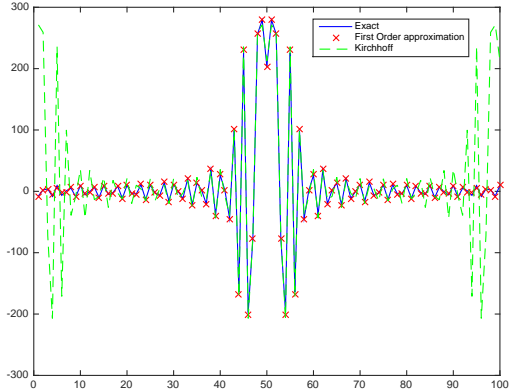} \quad 
\caption{Comparison of the exact solution of the problem \eqref{problem} with a Kirchhoff type approximation  \eqref{kir}
and a first order Bayliss-Turkel type approximation \eqref{ordre1} for the unit circle illuminated by a plane wave incidence with
$k=150$.}
\label{fig:modes}
\end{figure}

The high-frequency integral equation methods mentioned above use the 
asymptotic expansions developed  in the well-known paper by Melrose and
Taylor \cite{tay1} in the context of convex obstacles. From these expansions, an 
ansatz  is derived and incorporated into integral equations to eliminate the highly oscillatory part of the unknown which usually corresponds to the physical density, normal derivative of the total field, 
 computed on the surface
of the obstacles. This surface  
is decomposed into three regions, the illuminated and shadows regions in addition to
the deep shadow one. Each region is then numerically treated  differently and
the ansatz is set in general on the illuminated one. 
Although carefully designed, the aforementioned high-frequency integral
equation formulations result in ill-conditioned matrices 
 that limit the numerical
accuracy of the approximate solutions. 
One explanation lies in the fact that
the rapidly decaying behavior 
of the unknown density   
 in the deep shadow regions is
not incorporated into the approximation spaces as it is  not intrinsic to the
chosen ansatz. Generally speaking, it is not clear how to extract all the information needed from the leading term in the expansion given in \cite{tay1}, which restricts the construction of the ansatz.

In this paper, we derive new 
expansions of the normal derivative of the total field using
approximations of the Dirichlet to Neumann  (DtN) operator. 
The original expansions  employed a pseudo-differential 
decomposition of the DtN operator, and the related analysis focuses on the behavior  of this field  around  
the shadow boundary which leads to a  corrected formula for  the Kirchhoff approximation
 around this region  \cite{tay1}.   However, it has been shown that these expansions are valid in the entire surface of the obstacles \cite{victor,fatih1}.  Here, 
we choose first and second order approximations   of the DtN operator of the  Bayliss-Turkel type \cite{Bayliss}. 
To obtain these new expansions, we follow a similar procedure to the one given in \cite{tay1}.  
Briefly, it consists of first finding the kernel of a certain operator, which allows the computation of its amplitude,  
and then use  the stationary phase method to get the final expansions around the shadow boundary. In this case, we can use some
of the results derived by Melrose and Taylor \cite{tay1} in our analysis. The resulting expansions 
can then be used to appropriately build  an ansatz that contains the expected behavior of the solution in the three regions, namely,  the illuminated and deep shadow regions in addition to the  shadow boundaries.  This provides an improvement over the usual ansatz that behaves like Kirchhoff approximations, meaning that the corresponding  solutions are accurate  mostly  in   the illuminated regions.

 This paper is organized as follows. After reviewing the functional setting needed 
for this analysis, we state  the problem and explain  our  choice, regarding the approximation of the DtN operator, in the second section.  
The two following sections are, respectively, devoted to the derivation of  asymptotic expansions in the context of   
first and second order approximations of the DtN operator. The last section is reserved  for some conclusions.

In this work, we will use the following  functional spaces (for more details, see for instance \cite{Taylorbook,Chazarain}). 
Let $U$ be  an open bounded set of $\mathbb{R}^{n}$.
\begin{itemize}

\item $D(U)$: space of smooth test functions  with compact support, from $U$ to $\mathbb{R}^{n}$.
\item $D^{\prime }(U)$: space of distributions. 
\item $S(\mathbb{R}^{n})$: Schwartz space or space of rapidly decreasing functions on $\mathbb{R}^{n}$.
\item $S'(\mathbb{R}^{n})$: the space of tempered distributions, which is the dual space  of $S(\mathbb{R}^n)$. 
\item $\mathcal{E}^\prime$: space of compactly supported distributions. 
\item $\Psi^m$: space of pseudo-differential operators of order $m$.
\item $I^m$: space of Fourier integral operators of order $m$.
\end{itemize}

 We will also use  symbols of H$\ddot{\text{o}}$rmander's classes \cite{Hormander,Hormander1}, we say $p(x,\xi)\in S^m_{\rho,\delta}$ if and only if 
\begin{equation}
|D^\beta_xD^\alpha_\xi p(x,\xi)|\le C_{\alpha \beta}(1+|\xi|)^{m-\rho|\alpha|+\delta|\beta|}.
\end{equation}
In particular we say   $p(x,\xi)\in S^m$ if $p(x,\xi)\in S^m_{1,0}$. Note that each $p(x,\xi)$ admits an asymptotic expansion of the form
\begin{equation}
p(x,\xi)\sim \sum_{j\geq 0} p_j(x,\xi)
\end{equation}
for $|\xi|$ large where the $p_j(x,\xi)$ are homogeneous functions of degree $m-j$ in $\xi$.
Finally, if $p(x,\xi)\in S^m$, we say  $p(x,D)\in \Psi^m$, where $D$ is the corresponding pseudo-differential operator. 

\section{Model problem}

 Consider a convex obstacle $K \subset \mathbb{R}^{n+1}$ such that $B=\partial K\subset\mathbb{R}^{n+1}$  is a hypersurface and let $\Omega$ be the exterior domain given by $\Omega =\mathbb{R}^{n+1}\backslash K$.  We are interested  in solutions of the following wave equation
\begin{equation}\label{prob}
\left\{ \begin{array}{l}  
\displaystyle   (\partial^{2}_{t} -\Delta ) u(x,t)=0 \quad \mbox{in}\  
 \Omega \times \mathbb{R}, \\ [4pt] 
u(x,t)=-\delta(t-x \cdot \omega) = - u^{i} (x,t) \quad \text{on}\ B \times \mathbb{R},
\end{array}\right.
\end{equation}
where $u^i$ is the incident wave, $\omega$ denotes the incidence direction, $\delta$ is the Dirac function, and  $u\in D'(B\times \mathbb{R})$  \cite{tay1}. 
Defining the function $w$ by
\begin{equation}
w(x,k)=\int e^{ikt}u(x,t)dt\label{transfu}, 
\end{equation}
leads to the well-posed problem \cite{Wilcox}
\begin{equation}
\left\{ 
\begin{array}{ll}
(\Delta+k^{2})w(x,k)=0 \quad  \mbox{in}\ 
 \Omega \times \mathbb{R}, \\ [4pt]
w(x,k)=-e^{ikx \cdot \omega } \quad \text{on}\ B \times \mathbb{R},  \\ [4pt]
w(x,k)=\mathcal{O} (|x|^{-n/2})\ \mbox{and}\   \displaystyle (\partial_{|x|}-ik)w(x,k)=o (|x|^{-n/2})\ \mbox{for}\ |x|\rightarrow \infty.
\end{array}
\right.\label{problem}
\end{equation}
 For each $x\in \partial K=B$, ${\bf n}$ denotes the outgoing normal vector.
In what follows, we use the notation $w^t=w^s + w^i$ indicating the total field, where
$w^s$ is the scattered field solution of the problem \eqref{problem} and 
$w^i=e^{ikx \cdot \omega }$. In addition, we write the normal derivative of the total field as 
\begin{eqnarray}
\nonumber {\partial_{\mathbf n}} w^t= {\partial_{\mathbf n}} w^s +{\partial_{\mathbf n}} w^i\\[2pt]
 =a_Q (x,k)e^{ik x \cdot \omega},\hspace{-0.2cm}
 \end{eqnarray}
that we express with respect to the problem \eqref{prob}, and using the same notation given in  \cite{tay1}, by the operator
\begin{equation}
Q= \left (\mathrm{DtN} + (\omega \cdot {\bf n}) \partial_t) \right)F: \  \mathcal{E'}(B\times\mathbb{R})\rightarrow \mathcal{D'}(B\times \mathbb{R}),\label{OQ}
\end{equation}
where  $\mathrm{DtN} \in \Psi^1$ \cite{tay1} stands for the Dirichlet to Neumann operator 
\begin{equation}
\mathrm{DtN} :\mathcal{E}^{\prime }( B\times \mathbb{R})\rightarrow D^{\prime }(
B\times \mathbb{R}), \ \mathrm{DtN} u^i= -{\partial_{\bf n}}u|_{B \times \mathbb{R}},
\end{equation}
and $F$ is a  Fourier
integral operator to be defined later. Using a pseudo-differential decomposition of this operator DtN, Melrose and Taylor derived the well-known expansion 
\begin{equation}
\partial_{{\bf n}} w^t \sim  \sum_{p,l=0}^\infty k^{2/3-p-2l/3}a_{p,l}(\omega,x)\Psi^{(l)}(k^{1/3}Z(\omega,x))e^{ikx\cdot \omega}, \label{exporg}
\end{equation}
where $\Psi(\tau) \sim \sum_{j=0}^{\infty} c_{j} \tau^{1-3j}$ as $\tau \to \infty$, and it is rapidly
decreasing in the sense of Schwartz as $\tau \to -\infty$. The real-valued function $Z$ is
positive on the illuminated region, negative on the shadow region and vanishes precisely to
first order on the shadow boundary \cite{tay1}. Here, $a_{p,l}$ result from the application of the stationary phase method and the expansion of the 
symbol of the operator $Q$ \cite{tay1}.


\begin{remark}
As is mentioned in \cite{tay1},  the first term of  
the expansion \eqref{exporg} represents the classical Kirchhoff approximations.
 Indeed, if $\Psi(k^{1/3}Z(\omega,x))$ is replaced by the leading term in its asymptotic expansion 
\begin{equation}
\Psi(\tau)\simeq -2i\tau, \textsf{    for    }\tau \to +\infty, 
\end{equation}
 and taking  $
 a_{00}(\omega,x)=({\bf n} \cdot\omega) /{Z(\omega,x)}$  
 in the illiminated region $Z(\omega,x)>0$ 
  $({\bf n} \cdot \omega<0)$, we obtain   
  \begin{equation}\label{kir} 
    \partial_{{\bf n}} w^{t} \simeq k^{2/3}\frac{{\bf n} \cdot \omega}{Z(\omega,x)}(-2ik^{1/3}Z(\omega,x))e^{ikx \cdot \omega}=2ik {\bf n} \cdot \omega 
  e^{ikx \cdot \omega}. 
  \end{equation}
\end{remark}

Generally speaking, the aforementioned high-frequency integral equation methods, based on an  ansatz of the form $\partial_{{\bf n}} w^t=\eta(x) e^{ik x\cdot  \omega}$, 
although delivering better accuracy than the Kirchhoff approach in the illuminated region, they were designed  replicating
its behavior,  and this explains why the decay in the deep shadow region is not observed but somehow forced. 
This is partly due to the fact that an explicit form of the leading term in the asymptotic expansion \eqref{exporg}
is not available. We propose in this work to use  
approximations of the DtN
map to derive new asymptotic expressions of $\partial_{\bf n} w^t$. 
This can allow construction of a new    
 ansatz in order to improve the behavior and the accuracy of the solution  in the shadow and the deep shadow regions.  
 We use  first and second order approximations of the DtN operator given by Bayliss-Turkel \cite{Bayliss}  
\begin{eqnarray}
\label{ordre1} {\partial_{\bf n}}w^s(x,k) = -ikw^i(x,k)+\frac{c(x)}{2}w^i(x,k), \hspace{6cm} \\ [4pt]
{\partial_{\bf n}}w^s(x,k ) = -ikw^i(x,k)+\frac{c(x)}{2}w^i(x,k)-\frac{c(x)^{2}}{8(c(x)-ik)}w^i(x,k)-\frac{1}{2(c(x)-ik)}\partial _{x}^{2}w^i(x,k), \label{ordre2}
\end{eqnarray}
where $c(x)$ represents the curvature.  Although these conditions are approximations
of the DtN operator, we use the sign "=" for the sake of the presentation.

The motivation behind this choice of conditions \eqref{ordre1} and \eqref{ordre2} is summarized in the next example. 
Suppose that  $\Omega$ is a circle,  
in this case the exact solution  of problem \eqref{problem} is given by Bessel functions. 
We then compute the quantity
\begin{equation} \label{}
\partial_{\bf n} w^t =\partial_{\bf n} w^s +\partial_{\bf n} e^{ik x \cdot \omega}
\end{equation}
 using, (1)  the exact solution, (2)  the first order approximation \eqref{ordre1} approximating 
$\partial_{\bf n} w^s$,  and (3)  the approximations \eqref{kir} which corresponds exactly  to the  Kirchhoff approximation in the illuminated region. 
The resulting  calculations are  exhibited in Figure \ref{fig:modes}.
 We can observe that using the formula   \eqref{kir} produces a good approximation in
the illuminated region, but starts to degrade in the shadow boundaries to completely deteriorate in the deep shadow region in contrast with the  solution, based on condition \eqref{ordre1}, where we can observe a satisfactory approximation in those regions. This explains why in the Kirchhoff approximation
the solution is taken as zero  in the shadow  and deep shadow regions. However,  imposing this kind of constraints  leads to inaccurate solutions in the high frequency integral formulations.  

\begin{remark}\label{rem1}
	The  condition  \eqref{ordre1} was derived in the case of two and three dimensions while the condition \eqref{ordre2} was derived only in the two dimensional case \cite{Turkel1}.  Our computations do not distinguish between these two cases. However, at the end of this analysis, we  give an example of an expansion derived for a three dimensional second order approximation of the DtN operator. 
\end{remark}

%
%
%


\section{Expansion of the amplitude using the first order approximation}

 The results produced in this paper are based on some of the 
 results derived  in 
 the paper \cite{tay1}. 
  Our analysis starts by determining the kernel
 associated with the operator \eqref{OQ} in the case where the DtN operator is approximated by  \eqref{ordre1}.  
 For the sake of the
 presentation, we use the notation $\mathcal{C}(x) =c(x)/2$. 
  
\begin{theorem}
Let  $K\subset\mathbb{R}^{n+1}$ be a strictly convex bounded  obstacle such that $\partial K=B$, where 
$B$ is  $C^{\infty }$ hypersurface in  $\mathbb{R}^{n+1}$. Suppose that $ \Omega $ is  an open set of  $\mathbb{R}^{n+1}$ such that $\Omega =\mathbb{R}^{n+1}/K$, and let $w^s$ be a solution of (\ref{problem}).
Using the approximation (\ref{ordre1}), the operator
$Q$ \eqref{OQ} can be written as
\begin{equation}
Q=((1-{\bf n} \cdot \omega) \partial_t+\mathcal{C}(x))F: \ \mathcal{E'}(B\times\mathbb{R})\rightarrow\mathcal{D'}(B\times\mathbb{R}), \label{Qorder1}
\end{equation}
where $\kappa_Q (x,t)$ is its kernel given by
\begin{equation}
\kappa_Q (x,t)=((1-{\bf n} \cdot \omega) \partial_t+\mathcal{C}(x))\kappa_F(x,t),
\end{equation}
and $\kappa_F(x,t)=\delta(t-\omega \cdot x)$ is the kernel of the  Fourier integral operator $F$.
\end{theorem}

\begin{proof}
Using the approximation \eqref{ordre1} and the definition of the total field, we can write 
\begin{eqnarray}
{\partial_{\bf n}} w^t (x,k) 
= (-ik(1-{\bf n} \cdot \omega)+\mathcal{C}(x))e^{ikx\cdot \omega}.
\end{eqnarray}
To obtain the kernel of the operator $Q$, we compute the  Fourier transform with respect to the variable $t$. Let $\varphi_x(k)=\varphi(x,k) \in S(\mathbb{R}) $, we have then 
\begin{eqnarray}
\nonumber \langle \widehat{\partial_{\bf n}  w^t}(x,k),\varphi (x,k)\rangle_{S',S} 
=\left\langle \partial_{\bf n} w^t (x,k),\widehat{\varphi} (x,k)\right\rangle_{S',S} \hspace{0.8cm}\\[2pt]
\nonumber =\int_\mathbb{R}\partial_{\bf n} w^t(x,k)\widehat{\varphi} (x,k)dk \hspace{1cm} \\[2pt]
\nonumber=\int_{\mathbb{R}\times \mathbb{R}}\partial_{\bf n} w^t(x,k)\varphi (x,t)e^{-ikt}dk dt \hspace{-0.5cm} \\[2pt]
\nonumber=\int_{\mathbb{R}\times \mathbb{R}}\left[-ik(1-{\bf n} \cdot \omega)e^{ikx \cdot \omega}+\mathcal{C}(x) e^{ikx\cdot \omega}\right] \varphi(x,t)e^{-ikt}dkdt \hspace{-4.3cm}\\[2pt]
        \nonumber =\langle (\partial_t(1-{\bf n} \cdot \omega)+\mathcal{C}(x))\delta(t-\omega \cdot x),\varphi (x,t)  \rangle_{S',S} \hspace{-2.6cm}\\[2pt]
=\left\langle\kappa_Q(x,t),\varphi(x,t)\right\rangle_{S',S},\label{Souad1} \hspace{1.cm}
\end{eqnarray}
where $\langle\widehat{ik f(k)},\varphi\rangle_{S',S} = \langle-\partial_t\widehat{f}(t),\varphi\rangle_{S',S}$ and 
$\langle\widehat{e^{ik\omega \cdot x}},\varphi\rangle_{S',S}=\langle \delta (t-\omega \cdot x),\varphi\rangle_{S',S}$.  Therefore  
\begin{equation}
\kappa_Q(x,t)=((1-{\bf n} \cdot \omega)\partial_t+\mathcal{C}(x))\kappa_F(x,k),
\end{equation}
and $Q=((1-{\bf n} \cdot \omega) \partial_t+\mathcal{C}(x))F$. 
\end{proof}

In the following, we  use the same decomposition of the operator $F$ given in \cite{tay1} (equation (5.9)), that is, 
\begin{equation}
F=J(E_{1}\mathbb{A}_{i}+E_{2}\mathbb{A}_{i}^{\prime })K\label{F},
\end{equation}
with $E_{1}\in \Psi^{-n/2+1/6}$, $E_{2}\in \Psi ^{-n/2-1/6}$, and $J$ and $K$ are elliptic Fourier integral operators of order 0. The operators  $\mathbb{A}_{i}$ and $\mathbb{A}_{i}^{\prime }$ are Fourier 
Airy integral operators and are  defined by 
\begin{equation}
\mathbb{A}_{i}^{(l)}u(x,t)=\int e^{itk+ix\xi}A_i^{(l)}(\xi_1k^{-1/3})\widehat{u}(\xi ,k)d\xi dk, \label{Ai}
\end{equation}
\begin{equation}
\widehat{(\mathbb{A}_{i}u)}(\xi,k )=A_{i}(k^{-1/3}\xi _{1})\widehat{u}(\xi,k), \label{A}
\end{equation}
\begin{equation}
\widehat{(\mathbb{A}_{i}^{\prime }u)}(\xi,k)=A_{i}^{\prime }(k^{-1/3}\xi _{1})\widehat{u}(\xi,k), 
\end{equation}
where $\xi=(\xi_1,\dots, \xi_n) \in \mathbb{R}^{n}$, $k \in \mathbb{R}$, $l$ is an integer indicating the order of the derivative,  $A_{i}$ is the Airy function 
\begin{equation}
A_{i}(s)=\frac{1}{2\pi }\int\limits_{-\infty }^{+\infty}e^{i(\frac{t^{3}}{3}+st)}dt,
\end{equation}
with
\begin{equation}
A_{\pm }(s)=Ai(e^{\pm 2\pi i/3}s)\label{A+}, 
\end{equation}
and $A_{i}^{\prime }$ is its derivative. Finally, we define the operator $\mathbb{A}^{-1}$ \cite{tay1} as follows 
\begin{equation}
\widehat{(\mathbb{A}^{-1}u)}(\xi,k)=\frac{1}{A_+}(k^{-1/3}\xi _{1})\widehat{u}(\xi,k).  \label{invairy}
\end{equation}

The next theorem is concerned with the computation of the amplitude of the operator \eqref{Qorder1}.

\begin{theorem}\label{th4}
Let  $K$ and $J$ be elliptic Fourier  integral operators of order 0. Then the operator $Q$ and its kernel $\kappa_Q$ can be respectively written as
\begin{equation}
Q=(1-{\bf n} \cdot \omega)J\mathbb{A}^{-1}P_1K+\mathcal{C}(x) J \mathbb{A}^{-1}P_2K, \label{operateurQ1}
\end{equation}
\begin{equation}
\kappa_Q(x,t)=\int e^{i\psi_1(x,\xi,k)-ikt}\left((1-
{\bf n} \cdot \omega)a(x,\xi,k)+\mathcal{C}(x)  b(x,\xi,k)\right)\frac{1}{A_+}(k^{-1/3}\xi_1)d\xi dk, 
\end{equation}
such that  $P_1\in \Psi ^{-n/2+5/6}$ with a symbol $p_1$, $P_2\in \Psi^{-n/2-1/6}$ with a symbol $p_2$, 
$a(x,\xi,k)=p_1(x,\xi,k)a_J(x,\xi,k)$, 
$b(x,\xi,k)=p_2(x,\xi,k)a_J(x,\xi,k)$, with $a_J \in S_{(0,1)}^0$, 
and $\mathbb{A}^{-1}$ 
is a pseudo-differential operator defined by  \eqref{invairy} \cite{tay1}. 
In addition, its  amplitude is given by 
\begin{equation}\label{Rayhnahbel}
a_Q(x,k)=\int e^{ik\psi_1(x,\zeta)}\left((1-{\bf n} 
\cdot \omega)a_1(x,\zeta,k)+\mathcal{C}(x) b_1(x,\zeta,k)\right)\frac{1}{A_+}(k^{2/3}\zeta_1)d\zeta,
\end{equation}
 with $\zeta=k\xi$, $a_1(x,\zeta,k)=k^n a(x,k\xi,k )$ and $b_1(x,\zeta,k)=k^n b(x,k\xi,k)$. 
\end{theorem} 

\begin{proof}
Using (\ref{F}) and equation \eqref{Qorder1} we obtain  
\begin{equation}
Q=(1-{\bf n} \cdot \omega)J(E_{3}\mathbb{A}_{i}+E_{4}\mathbb{A}_{i}^{\prime })K+\mathcal{C}(x) J( E_{1}\mathbb{A}_{i}+E_{2}\mathbb{A}_{i}^{\prime})K, 
\end{equation}
 where $E_{1}\in \Psi ^{-n/2+1/6}$, $E_{2}\in \Psi^{-n/2-1/6}$,
 $E_{3}\in\Psi^{-n/2+1/6+1}$, and $E_{4}\in \Psi^{-n/2-1/6+1}$. The operators  $\mathbb{A}_{i}$ and  $\mathbb{A}_i^{\prime }$ 
are Airy operators given by \eqref{Ai}.  Using Theorem 6.5 in \cite{tay1}, we can write  
\begin{equation}
 Q=(1-{\bf n} \cdot \omega)J\mathbb{A}^{-1}P_1K+\mathcal{C}(x)  J\mathbb{A}^{-1}P_2K, \label{operateurQ2}
 \end{equation}
 with $P_1\in\Psi^{-n/2-1/6+1}$ and $P_2\in\Psi^{-n/2-1/6}$. To compute the  oscillatory integral related to $Q$, consider $\varphi(x,t) \in 
 S(\mathbb{R}^n \times \mathbb{R}) $ and the 
  Dirac delta function 
 $\delta _{(x,t)}\in \mathcal{E} ^{\prime }(B\times\mathbb{R})$  used here to find the kernel of  $Q$ at the base point \cite{tay1,MonoMelrose},  
  we have then 
\begin{eqnarray*}
\langle Q\delta _{0},\varphi (x,t)\rangle &=&\langle(1-{\bf n} \cdot \omega))J\mathbb{A}^{-1}P_1K\delta_{0} +\mathcal{C}(x) J\mathbb{A}^{-1}P_2K\delta_{0},\varphi (x,t)\rangle \nonumber\\&=&\langle(1-{\bf n} \cdot \omega)J\mathbb{A}^{-1}P_{3}\delta _{0},\varphi (x,t)\rangle +\langle  \mathcal{C}(x) J\mathbb{A}^{-1}P_{4}\delta _{0},\varphi (x,t)\rangle \nonumber \\&=&\langle(1-{\bf n} \cdot \omega)JP_3\mathbb{A}^{-1}\delta _{0},\varphi (x,t)\rangle +\langle \mathcal{C}(x) JP_{4}\mathbb{A}^{-1}\delta _{0},\varphi (x,t)\rangle \nonumber \\
&=&\langle Q_{1}\delta_{0},\varphi (x,t)\rangle +\langle Q_{2}\delta_{0,},\varphi (x,t)\rangle, 
\end{eqnarray*}
where
\begin{equation}
  Q_{1}\delta_{0}=(1-{\bf n} \cdot \omega)JP_3\mathbb{A}^{-1}\delta _{0}, \quad 
 Q_{2}\delta_{0}=(1-{\bf n} \cdot \omega)JP_4\mathbb{A}^{-1}\delta _{0},
\end{equation}
and
\begin{equation}
	J u(x,t)=\int e^{i\psi_1(x,\xi,k)-iy\xi-ikt_1} a_J(x,y,t_1\xi,k)u(y,t_1) dy dt_1d\xi dk,
\end{equation}
with  $P_{3}\in \Psi ^{-n/2-1/6+1}$ and $P_{4}\in \Psi ^{-n/2-1/6}$ such that $P_{3}=P_1K$ and  $P_{4}=P_2K$, and knowing that 
 $\mathbb{A}^{-1}$ commute with $P_3$ and $P_4$ \cite{tay1}. Here, $(y,t_1)\in \partial K \times \mathbb{R}$, $(\xi,k)$ indicates  the dual of $(x,t)$, 
and the phase function $\psi_1$ is defined in the three regions of the obstacle  \cite{tay1}. In the illuminated region $\{x \in \partial K, \ {\bf n}(x)\cdot \omega < 0\}$, it is given by 
 \begin{equation}
 \psi_1(x,\xi,k)=-\frac{|\xi'|^2}{2k}-\frac{|x'|^2}{2}k+\frac{3}{2}(-\xi_1 k^{-1/3})^{3/2}, 
\end{equation} 
 while in the shadow region $\{x \in \partial K, \ {\bf n} (x)\cdot \omega >0\}$, we have 
  \begin{equation}
 \psi_1(x,\xi,k)=-\frac{|\xi|'^2}{2k}-\frac{|x'|^2}{2}k-\frac{3}{2}(-\xi_1 k^{-1/3})^{3/2}. 
\end{equation}  
 Finally, on the shadow boundary $\{x\in \partial K,  \ {\bf n}(x)\cdot \omega =0\}$,     the phase function is as follows 
 \begin{equation}
 \psi_1(x,\xi,k)=-\frac{|\xi'|^2}{2k}-\frac{|x'|^2}{2}k
 \end{equation}
 since $\xi_1=0$ 
 \cite{tay1}. Here,   $x'=(x_2,...,x_n)$, $\xi'=(\xi _2,...,\xi _n)$ such that $x \in \partial K ,\xi \in \mathbb{R}^n$, $t,k \in \mathbb{R}$, and $a_J(x,t,\xi,k)\in S^0_{1,0}$.
The pseudo-differential operator $P_3$ is defined by 
\begin{equation}
P_3u(y,t_1)=\int e^{i(y-Y)\eta+i(t_1-T)\tau}p_3(y,t_1,Y,T,\eta,\tau)u(Y,T)dY dT d\eta d\tau, 
\end{equation}
where $(\eta,\tau)\in \mathbb{R}^n\times \mathbb{R}$ is the dual couple of $(y,t_1)$  and $(Y,T)\in \partial K \times \mathbb{R}$.
Our objective is to compute $Q_1\delta_{0}=JP_3\mathbb{A}^{-1}\delta_{0}$. First,  using standard 
   calculations on composition of operators \cite{Hormander, AndreMartinez}, we have 
\begin{eqnarray}
J\circ P_3 u(x,t)
=\int e^{i\psi_1(x,\xi,k) -iY\xi-iTk}p_{_{J\circ P_3}}(x,Y,T,\xi ,k)u(Y,T)dY dT d\xi dk, 
%
\end{eqnarray}
with 
\begin{eqnarray}
p_{_{J\circ P_3}}(x,Y,T,\xi ,k)=a(x,Y,T,\xi,k)\nonumber \hspace{7.3cm}\\=\int e^{i(Y-y)(\xi - \eta)+i(T-t_1)(k-\tau)}a_J(x,y,t_1,\xi,k)p_3(y,t_1,Y,T,\eta,\tau)dy dt_1 d\eta d\tau, \hspace{-0.9cm} 
\end{eqnarray}
thus
\begin{equation}
J\circ P_3 u(x,t)=\int e^{i\psi_1(x,\xi,k)}a(x,\xi,k)\widehat{u}(\xi,k)d\xi dk, \label{J0P}
\end{equation}
where $a(x,\xi,k)=p_3(x,\xi,k)a_J(x,\xi,k)$. To find $Q_1\delta_{0}$ we need to replace $u(x,t)$ by $\mathbb{A}^{-1}\delta_{0}$ in (\ref{J0P}) where 
\begin{equation}
\widehat{\mathbb{A}^{-1}\delta_{(0,t)}}(\xi ,k)=\frac{1}{A_+} (k^{-1/3}\xi_1) e^{-itk}. 
\end{equation}
 We get 
\begin{eqnarray}
 Q_{1}\delta_{0}=(1-{\bf n} \cdot \omega)J\circ P_3 \circ\mathbb{A}^{-1}\delta_{0}\nonumber \hspace{3.2cm}\\=\int e^{i\psi_1(x,\xi,k)}(1-{\bf n} \cdot \omega)a(x,\xi,k)\widehat{\mathbb{A}^{-1}\delta_{0}}(\xi,k)d\xi dk \hspace{-0.3cm} \nonumber \\ 
 =\int e^{i\psi_1(x,\xi,k)-ikt}(1-{\bf n} \cdot\omega)a(x,\xi,k)\frac{1}{A_+} (k^{-1/3}\xi_1) d\xi dk.  \hspace{-0.8cm} \label{qq1}
\end{eqnarray}
A similar approach for $Q_2$ gives  
\begin{equation}\label{qq2}
Q_2\delta _{0}=\int e^{i\psi _{1}(x,\xi,k )-ikt} \mathcal{C}(x) b(x,\xi,k ) \frac{1}{A_{+}} (k^{-1/3}\xi _{1}) d\xi dk. 
\end{equation}
The equation \eqref{operateurQ2} becomes 
\begin{equation}
 Q\delta _{0} =\int e^{i\psi _{1}(x,\xi,k )-ikt} \left[
(1-{\bf n} \cdot\omega)a(x,\xi,k )+ \mathcal{C}(x)  b(x,\xi,k )\right]\frac{1}{A_{+}}(k^{-1/3}\xi_{1}) dk d\xi, 
\label{Q}
\end{equation}
where $b(x,\xi,k)\in S_{1,0}^{-n/2-1/6}$  and 
$a(x,\xi,k)\in S_{1,0}^{-n/2+5/6}$. This shows that the kernel  $\kappa_Q$ is as follows
\begin{equation}
\kappa _{Q}(x,t)=\int e^{i\psi_1(x,\xi,k )-ikt}\left[(1-{\bf n} \cdot \omega)
a(x,\xi,k )\frac{1}{A_{+}}(k^{-1/3}\xi _{1})+\mathcal{C}(x)  b(x,\xi,k )\frac{1}{A_{+}}(k^{-1/3}\xi _{1})\right] d\xi dk. 
\end{equation}

Taking now the inverse Fourier transform of $\kappa_Q$, we obtain 
the following  amplitude  
\begin{equation}
a_{Q}(x,k)=\int e^{i\psi_1(x,\xi,k) } [(1-{\bf n} \cdot \omega)a(x,\xi,k)\frac{1}{A_{+}}(k^{-1/3}\xi _{1})+\mathcal{C}(x) b(x,\xi,k )\frac{1}{A_{+}}(k^{-1/3}\xi _{1})] d\xi. 
\end{equation}

Applying the change of variable  $\zeta=k\xi$ with 
$\zeta\in \mathbb{R}^{n}$, we find   
\begin{equation}
a_{Q}(x,k)=\int e^{ik\psi_2(x,\zeta)}\left[(1-{\bf n} \cdot \omega)a_{1}(x,\zeta,k)\frac{1}{A_{+}}(k^{2/3}\zeta_{1})+\mathcal{C}(x) b_{1}(x,\zeta,k )\frac{1}{A_{+}}(k^{2/3}\zeta_{1})\right] d\zeta,  \label{aQ} 
\end{equation}
such that 
$a_1(x,\zeta,k)=k^na(x,k\zeta,k)$, $b_1(x,\zeta,k)=k^nb(x,k\zeta,k)$  and $\psi_2(x,\zeta)=\psi_1(x,\zeta,1)=-\frac{\xi'^2}{2}-\frac{x'^2}{2}$.
\end{proof}

The remaining part  of the computation of the
  asymptotic expression of $a_Q$ \eqref{aQ} consists of
  applying the stationary phase method. First, we need the next lemma \cite{tay1}.   
  
\begin{lemma} 
The function $\Psi \in S^1 (\mathbb{R})$ defined as follows  
\begin{equation}
\Psi(\tau)=e^{-i\tau^3/3}\int \frac{1}{A_+(s)}e^{-is\tau}ds
\end{equation}
is rapidly decreasing for $\tau\rightarrow -\infty$, where 
 ${A_+(s)=A_i(e^{\frac{2\pi i}{3}}s})$.
\end{lemma}

\begin{theorem}
The asymptotic expansion of the amplitude $a_Q$  is given by 
\begin{equation}\label{sum.aQ}
a_Q(x,k) = \sum_{p,l=0}^{P,L} k^{2/3-p-2l/3}\left((1-{\bf n} \cdot \omega)a_{p,l}(\omega,x)+ \mathcal{C}(x) b_{p,l}(\omega,x)\right)\psi^{(l)}(k^{1/3}Z(\omega,x))+R_{P,L}(k), 
\end{equation}

such that 
\begin{equation}
|R_{P,L}(k)|\leq C_{PL} k^{-min(2L/3,P+1/3)}, 
\end{equation}
and where  $p \in \{0,1..,P\}$, $l \in \{0,1..,L\}$, $C_{PL}$ is a constant depending on $L$ and $P$,  $\omega$ is the incidence direction,  and $Z(\omega,x)$ is a continuous real function  that is positive on the illuminated region, negative on the shadow region, and vanishing on the shadow boundary. The functions $a_{p,l}$ and $b_{p,l}$ result from the expansion of the symbols $a_1$ and $b_1$ (see Theorem \ref{th4}) and the application of the stationary phase method.  
\end{theorem}
\begin{proof} 
First, let us note that
\begin{eqnarray}
\frac{1}{A_+}(k^{2/3}\zeta_1)=\mathcal{F}^{-1} \left ( \widehat{\frac{1}{A_+}(k^{2/3}\zeta_1)} \right ) \nonumber\\\hspace{-2cm}=
\mathcal{F}^{-1}\left(\int e^{-ikt\zeta_1}\frac{1}{A_+}(k^{2/3}\zeta_1)d\zeta_1 \right)\nonumber \hspace{-2cm}\\
=\mathcal{F}^{-1}\left(e^{ik\frac{t^3}{3}}\Psi(k^{1/3}t)\right)\nonumber \hspace{-0.1cm}\\
\label{Assam} =k^{1/3}\int e^{ikt\zeta_1 +ik\frac{t^3}{3}}\Psi(k^{1/3}t)dt. \hspace{-1.2cm}
\end{eqnarray}
Using \eqref{Assam} and \eqref{aQ},   $a_Q$ becomes  
\begin{eqnarray}
a_{Q}(x,k)=k^{1/3}\int e^{ik\psi_2(x,\zeta)+ikt\zeta_1 +ik \frac{t^3}{3}}\left[(1-{\bf n} \cdot \omega)a_{1}(x,\zeta ,k)+\mathcal{C}(x)b_{1}(x,\zeta ,k )\right] \Psi(k^{1/3}t)dt d\zeta, \label{aQ2}
\end{eqnarray}
where $a_1(x,\zeta,k)\in S^{n/2+7/6}$ and $b_1(x,\zeta,k)\in S^{n/2+1/6}$ are given in Theorem \ref{th4}. Let us assume now that  
\begin{equation}
(1-{\bf n} \cdot \omega)a_{1}(x,\zeta,k)+\mathcal{C}(x) b_{1}(x,\zeta,k)=A_1(x,\zeta,k) \in S^{n/2+7/6}. 
\end{equation} 
Using properties of pseudo-differential operators \cite{AndreMartinez}, we can write
\begin{equation}
A_1(x,\zeta,k)=\sum^{P}_{p=0}k^{n/2+7/6-p}A_p(x,\zeta), \label{A_1}
\end{equation}
with $P=n/2+7/6$, and 
\begin{equation}
  A_p(x,\zeta)=(1-{\bf n} \cdot \omega)a_{p}(x,\zeta)+\mathcal{C}(x) b_{p}(x,\zeta). 
\end{equation}
This shows that the integral \eqref{aQ2} can be rewritten as
\begin{equation}
k^{1/3}\sum^{P}_{p=0}k^{n/2+7/6-p} \int e^{ikf(\zeta,t)} A_p(x,\zeta)\Psi(k^{1/3}t)dt d\zeta, \label{aQ2-2} 
\end{equation}
such that 
\begin{equation}
f(\zeta ,t)=\psi_1(x,\zeta )+t\zeta_{1}+\frac{t^3}{3},
\end{equation}
see \cite{tay1} for more details regarding the definition of $f$. 
To get the asymptotic expansion of $a_Q$, it remains to apply the stationary phase method to 
 \begin{equation}
 \int e^{ikf(\zeta,t)} A_p(x,\zeta)\Psi(k^{1/3}t)dt d\zeta. 
\end{equation}
The conditions for this application are satisfied and the 
computation of the critical points is done in \cite{tay1}.  
Using standard calculations regarding the  stationary phase  method \cite{AndreMartinez,ClaudeZuily02}, we obtain 
  \begin{eqnarray}
   \int e^{ikf(\zeta,t)} A_p(x,\zeta)\Psi(k^{1/3}t)dt d\zeta = k^{-1/3}\sum_{l=0}^{L} k^{-2l/3-(n+1)/2} A_{p,l}(\omega,x)\Psi^{(l)}(k^{1/3}Z(\omega,x))+R_{L}(k)\label{A_pl}, 
  \end{eqnarray}
  where $ A_{p,l}(\omega,x)=\partial^{l}_{\zeta_1} A_{p}(x,\zeta)|_{\zeta=0}$, and 
  \begin{equation}
 |R_L(k)|\leq C_L k^{-n/2-2L/3-3/2},
\end{equation}   
 with $C_L$ a constant. 
  Using (\ref{A_pl}) in (\ref{aQ2-2}), we obtain 
\begin{equation}
a_Q(x,k)= \sum_{p,l=0}^{P,L} k^{2/3-p-2l/3}\left((1-{\bf n} \cdot \omega)a_{p,l}(\omega,x)+\mathcal{C}(x) b_{p,l}(\omega,x)\right)\Psi ^{(l)}(k^{1/3}Z(\omega,x)) +R_{P,L}(k),
\end{equation}
 with  
\begin{equation}
|R_{P,L}(k)|\leq C_{PL} k^{-min(2L/3,P+1/3)}, 
\end{equation}
and $l \in \{0,1..,L\}$ and $P$ is a real number. 
\end{proof}

The next theorem establishes a relation between the functions $a_{p,l}$ and  $b_{p,l}$ found in \eqref{sum.aQ}. 

\begin{theorem}\label{thab}
The functions $a_{p,l}$ and  $b_{p,l}$ in \eqref{sum.aQ},
$p,l\geq 0$, 
satisfy the equation    
\begin{equation}
b_{p,l}(\omega,x)=-\frac{a_{p,l}(\omega,x)}{ik}. 
\end{equation} 
\end{theorem}

\begin{proof}
From  equation \eqref{Qorder1} and some operators used  in Theorem \ref{th4} ($F$, $J$, $P_4$ and $P_3$), we have      
\begin{equation}
Q=((1-{\bf n} \cdot \omega)\partial_tF+\mathcal{C}(x)F, \quad Fu(x,t)=JP_{4}\mathbb{A}^{-1}u(x,t). 
\end{equation} 
Similar calculations to the equation \eqref{qq2} give 
\begin{equation}
Fu(x,t)=\int e^{i\psi_1(x,\xi,k)-ikt-ix\xi}b(x,\xi,k)\frac{1}{A_+} (k^{-1/3}\xi_1) \widehat{u}(\xi,k)d\xi dk,
\end{equation}
with $b(x,\xi,k)=a_J(x,\xi,k)p_4(x,\xi,k)\in S^{-n/2-1/6}$, where 
$a_J(x,\xi,k)$ and $p_4(x,\xi,k)$ are given in Theorem \ref{th4}. 
 Knowing that $\partial_tFu(x,t)=JP_3\mathbb{A}^{-1}u(x,t)$ and 
\begin{eqnarray}
\partial_tFu(x,t)=\partial_t\int e^{i\psi_1(x,\xi,k)-ikt-ix\xi}b(x,\xi,k)\frac{1}{A_+} (k^{-1/3}\xi_1) \widehat{u}(\xi,k)d\xi dk,\nonumber\hspace{0.5cm} \\=\int e^{i\psi_1(x,\xi,k)-ikt-ix\xi}(-ik)b(x,\xi,k)\frac{1}{A_+} (k^{-1/3}\xi_1)\widehat{u}(\xi,k)d\xi dk,
\end{eqnarray}
implies that $a(x,\xi,k)=-ikb(x,\xi,k)$. 
%
We know that $b(x,\xi,k)\in S^{-n/2-1/6}$, and then usual pseudo-differential calculus results in $a(x,\xi,k)\in S^{-n/2+5/6}$.
This allows us to conclude that $a_{p,l}(\omega,x)=-ik b_{p,l}(\omega,x)$. 
\end{proof}

\subsection{Some estimates  of the asymptotic expansion \eqref{sum.aQ}}

Two estimates are established in this subsection. For the completion 
of the paper, we recall the next lemma \cite{tay1}.

\begin{lemma}
Let $\Psi \in S^1(\mathbb{R})$,  $\Psi(\tau)$ 
is rapidly decreasing for 
 $\tau\rightarrow -\infty$
and 
\begin{equation}
\Psi(\tau) \sim \sum_{j=0}^\infty c_j\tau^{1-3j} \quad \mbox{for}\ \tau  \rightarrow +\infty. 
\end{equation}
\end{lemma}


The next result compares 
the asymptotic expansion \eqref{sum.aQ} with $\partial_{\bf n} w^t$. 

\begin{proposition}
If $a_Q$ is the amplitude given by (\ref{sum.aQ}), then
\begin{equation}
|\eta (x,k)-{\partial_{\bf n}} w^t (x,k) |\leq  C k^{-1} {\text{         for           }} k \rightarrow + \infty, 
\end{equation}
where $C$ is a real constant and 
\begin{equation}\label{eta} 
\eta(x,k) = 
\sum_{p,l=0}^{P,L} k^{2/3-p-2l/3}\left((1-{\bf n} \cdot \omega)a_{p,l}(\omega,x)+ \mathcal{C}(x) b_{p,l}(\omega,x)\right)\psi^{(l)}(k^{1/3}Z(\omega,x))e^{ik x \cdot \omega}. 
\end{equation}
\end{proposition}
\begin{proof}
We know that 
\begin{equation}
{\partial_{\bf n}} w^t (x,k)= -ik(1-{\bf n} \cdot \omega) e^{ikx\omega}+\mathcal{C}(x)e^{ikx \cdot \omega}. 
\end{equation}
Using Theorem \ref{thab},  the expansion \eqref{eta} can be written as  
\begin{eqnarray}
\eta(x,k) \sim k^{2/3}\left((1-{\bf n} \cdot \omega)-\frac{\mathcal{C}(x)}{ik}\right)a_{00} (\omega,x)\Psi(k^{1/3}Z(\omega,x))e^{ikx \cdot \omega}\nonumber\hspace{2.5 cm} \\+\sum_{p,l=1}^{P,L} k^{2/3-p-2l/3}\left((1-{\bf n} \cdot \omega)-\frac{\mathcal{C}(x)}{ik}\right)a_{pl}(\omega ,x)\Psi ^{(l)}(k^{1/3}Z(\omega,x))e^{ikx \cdot \omega} \hspace{0.5 cm}\nonumber\\ = k^{2/3}\left((1-{\bf n} \cdot \omega)-\frac{\mathcal{C}(x)}{ik}\right)a_{00}(\omega,x)\Psi(k^{1/3}Z(\omega,x))e^{ikx \cdot \omega} \nonumber \hspace{2.5 cm }\\ +\sum_{\beta, \alpha =0}^{P-1,L-1} k^{-1-\beta -2\alpha/3}\left((1-{\bf n} \cdot \omega)-\frac{\mathcal{C}(x)}{ik}\right)a_{\beta +1,\alpha +1}(\omega,x)\Psi ^{(\alpha +1)}(k^{1/3}Z(\omega,x))e^{ikx \cdot \omega}.\hspace{-1 cm}
\end{eqnarray}
From  the preceding lemma, we have $\Psi (k^{1/3}Z(x,\omega))\sim -ik^{1/3} Z(x,\omega)$ for  $k \rightarrow + \infty$
and taking $a_{00}(x,\omega)=\frac{1}{Z(x,\omega)}$ \cite{tay1}, we obtain  
\begin{eqnarray}
\eta(x,k) \sim \left(-ik(1-{\bf n} \cdot \omega)+\mathcal{C}(x)\right)e^{ikx \cdot \omega}\nonumber \hspace{10.cm}\\+\sum_{\beta, \alpha =0}^{P-1,L-1} k^{-1-\beta -2\alpha/3}\left((1-{\bf n} \cdot \omega)-\frac{\mathcal{C}(x)}{ik}\right)a_{\beta +1,\alpha +1}(\omega,x)\Psi ^{(\alpha +1)}(k^{1/3}Z(\omega,x))e^{ikx \cdot \omega}. \hspace{1.6 cm}
\end{eqnarray}
Knowing that   $ |\partial^{\alpha +1}_{\tau} \Psi (\tau)|\leq M_{\alpha} \tau ^{-\alpha}$ for each $\tau \in \mathbb{R}$ and $|1-{\bf n} \cdot \omega|\leq 2$, we get 
\begin{eqnarray}
\left|\eta(x,k)-{\partial_{\bf n} w^t(x,k)} \right| \leq \sum_{\beta, \alpha =0}^{P-1,L-1} \left| k^{-1-\beta -2\alpha/3}\left((1-{\bf n} \cdot \omega)-\frac{\mathcal{C}(x)}{ik}\right)a_{\beta +1,\alpha +1}(\omega,x)\Psi ^{(\alpha +1)}(k^{1/3}Z(\omega,x))\right| \nonumber \hspace{-0.5cm} \\ \leq \sum_{\beta, \alpha =0}^{P-1,L-1}M  k^{-1-\beta -\alpha} \left|2+ \frac{\textsf{max}_{x\in B} \mathcal{C}(x)}{k}\right|\nonumber \hspace{5.5cm}\\ \leq C k^{-1}, \hspace{10.5cm}
\end{eqnarray}
and $C$ is a real constant.
\end{proof}

The following result estimates  \eqref{sum.aQ} near the shadow boundary. 
\begin{proposition}
If $a_Q$ is the amplitude given by \eqref{sum.aQ}, then
\begin{equation}\label{es2}
|a_Q(x,k)|\leq Mk^{2/3}  \ \  \mbox{ for}\  k \rightarrow + \infty
\end{equation}
where $M$  is a real constant.
\end{proposition}
\begin{proof}
  From the definition of $a_Q(x,k)$, it follows that 
\begin{equation}
|a_Q(x,k)| \leq \sum_{p, l =0}^{P,L} k^{2/3-p -2l/3}\left|\left((1-{\bf n} \cdot \omega)-\frac{\mathcal{C}(x)}{ik}\right)\right| \left| a_{p,l}(\omega,x)\Psi ^{(l)}(k^{1/3}Z(\omega,x))\right|,\hspace{-1 cm}
\end{equation}
with $p\in \{0,1,...,P\}$ and  $l\in \{0,1,...,L\}$. 
We know that in the shadow boundary $|1-{\bf n} \cdot \omega|\leq 1$. We assume that the curvature is constant, so $\mathcal{C}(x)=C$,   then 
\begin{eqnarray}
|a_Q(x,k)| \leq  (k^{2/3}+Ck^{-1/3})\left|a_{0,0}(\omega,x)\Psi(k^{1/3}Z(\omega,x))\right| \nonumber \hspace{2.5cm}\\ +\sum_{\alpha, \gamma =0}^{P-1,L-1} k^{-1-\alpha - 2\gamma/3}(1+ Ck^{-1}) \left| a_{\alpha +1,\gamma +1}(\omega,x)\Psi ^{(\gamma +1)}(k^{1/3}Z(\omega,x))\right|. \hspace{-2 cm}
\end{eqnarray}
The function  $\Psi$ and all its derivatives are bounded,   
\begin{equation}
|a_{0,0}(\omega,x)\Psi(k^{1/3}Z(\omega,x))|\leq M_1, 
\end{equation}
 for all $\gamma \in N$,   
 where $M_1$ is a real constant.
  Thus  
\begin{eqnarray}
|a_Q(x,k)|\leq M k^{2/3}+\sum_{\alpha,\gamma =0}^{P-1,L-1} k^{-1-\alpha -2\gamma/3}(1+Ck^{-1})|a_{\alpha +1,\gamma +1}(\omega,x)|M_{\gamma},\hspace{2.4cm}
\end{eqnarray}
such that $M=(1+C)M_1$  and $|\Psi^{(\gamma+1)}(\tau)(k^{1/3}Z(\omega,x))|\leq M_{\gamma}$. Taking  $k\rightarrow + \infty$, we obtain \eqref{es2}
\end{proof}

\section{Expansion of  the amplitude  using the second order approximation}

We are now interested in computing the asymptotic expansion in the case where we have 
\begin{eqnarray}
{\partial_{\bf n} w^s(x,k)} = -ikw^i(x,k)+\frac{c(x)}{2}w^i(x,k)-\frac{c(x)^{2}}{8(c(x)-ik)}w^i(x,k)-\frac{1}{2(c(x)-ik)}\partial^{2}_{x}w^i(x,k) \nonumber\\[4pt]
= -ikw^i(x,k)+\frac{c(x)}{2}w^i(x,k)-\frac{1}{2(c(x)^{2}+k^{2})}(c(x)+ik)(\frac{c(x)^{2}}{4}+ \partial^{2}_{x})w^i(x,k). \label{appr02}
\end{eqnarray}

As for the first order case, we compute first the kernel associated to the operator $Q$ \eqref{OQ} in the case where the DtN is approximated by \eqref{appr02}. For the sake of simplicity, we denote  the operator \eqref{Qorder1} by $Q_1$. 

\begin{theorem}
Let  $K\subset\mathbb{R}^{n+1}$ be a strictly convex bounded  obstacle such that $\partial K=B$, where 
$B$ is a  $C^{\infty }$ hypersurface in  $\mathbb{R}^{n+1}$. Suppose that $ \Omega $ is an open set of  $\mathbb{R}^{n+1}$ such that $\Omega =\mathbb{R}^{n+1}/K$. Let $w^s$ be a solution of (\ref{problem}). Using the approximation (\ref{appr02}), the operator $Q$ can be written as
\begin{eqnarray}
Q=Q_{1}- \frac{\pi}{2c(x)}T_{e^{-c(x)|t|}}(c(x)-\partial_t)\tilde{F} \label{Qapp2},
\end{eqnarray}
where $\kappa_Q(x,t)$ is its kernel given by 
\begin{eqnarray}
\kappa_Q(x,t)=((1-{\bf n} \cdot \omega)+\frac{c(x)}{2})\kappa_F(x,t)-\frac{\pi}{2c(x)}  e^{-c(x)|t|}\ast \left(c(x)-\partial_t\right)\kappa_{\tilde{F}}(x,t),
\end{eqnarray}
and $T_{e^{-c(x)|t|}}$ denotes the convolution operator of $e^{-c(x)|t|}$,
  $\displaystyle \tilde{F}=(\frac{c^2(x)}{4}+\partial^{2}_{x})F$, and 
 $\displaystyle \kappa_{\tilde{F}}(x,t)=(\frac{c^2(x)}{4}+\partial^{2}_{x})\kappa_F(x,t)$ where $\kappa_F(x,t)=\delta(t-\omega \cdot x)$. 
\end{theorem}
\begin{proof} 
Using the approximation \eqref{appr02} and the definition of the total field, we can write   
\begin{eqnarray}
\partial_{\bf n} w^{t}(x,k) =\left (-ik(1-{\bf n} \cdot \omega)+\frac{c(x)}{2}\right )
 e^{ikx\cdot \omega}-\frac{1}{2(c(x)^{2}+k^{2})}(c(x)+ik)(\frac{c(x)^{2}}{4}+\partial^{2}_{x} )e^{ikx \cdot \omega}\nonumber \hspace{0.2cm}\\=\mathcal{Q}_1 (x,k) +\mathcal{Q}_2 (x,k),  \hspace{8.8cm}
\end{eqnarray}
with 
\begin{equation}
 \mathcal{Q}_1 (x,k) =\left (-ik(1-{\bf n} \cdot \omega)+\frac{c(x)}{2} \right )
 e^{ikx\cdot \omega},
\end{equation} 
\begin{equation}
\mathcal{Q}_2 (x,k) =\left ( -\frac{1}{2(c(x)^{2}+k^{2})}(c(x)+ik)(\frac{c(x)^{2}}{4}+\partial_{x}^{2}) \right ) e^{ikx\cdot \omega}.
\end{equation}
To obtain the kernel of the operator $Q$, we compute  the Fourier transform  of the amplitude $\partial_n w^t(x,k)$ with respect to $k$. Let $\varphi_x(k)=\varphi(x,k)\in S(\mathbb{R})$, thus
\begin{eqnarray}
\left\langle \widehat{\partial_{n} w^t(x,k)},\varphi\right\rangle_{S',S}=\left\langle \widehat{\mathcal{Q}_1 (x,k)},\varphi\right\rangle_{S',S}+\left\langle \widehat{\mathcal{Q}_2 (x,k)},\varphi\right\rangle_{S',S}. \hspace{-0.5cm}
\end{eqnarray}
The quantity $\left\langle \widehat{\mathcal{Q}_1 (x,k)},\varphi\right\rangle_{S',S}$ is given by \eqref{Souad1}. For the one regarding $\mathcal{Q}_2 (x,k)$, we have
\begin{eqnarray}
\nonumber \left\langle \widehat{\mathcal{Q}_2 (x,k)},\varphi\right\rangle_{S',S}=\left\langle \mathcal{Q}_2 (x,k),\widehat{\varphi}\right\rangle_{S',S} \hspace{7.3cm}\\ [2pt] \nonumber
 =-\int_{\mathbb{R}\times\mathbb{R}}\left[\frac{1}{2(c(x)^{2}+k^{2})}(c(x)+ik)(\frac{c(x)^{2}}{4}+\partial_{x}^{2})e^{ikx \cdot \omega}\right] \varphi_x(t)e^{-ikt}dtdk  \hspace{-0.7cm}\\ [2pt] \nonumber =-\int _{\mathbb{R}}\left[\int_{\mathbb{R}}\frac{1}{2(c(x)^{2}+k^{2})}e^{-ikt}dk \ast \int_{\mathbb{R}}(c(x)+ik)(\frac{c(x)^{2}}{4}+\partial_{x}^{2})e^{ikx \cdot 
   \omega}e^{-ikt}dk \right]\varphi_x(t)dt  \hspace{-2.7cm} \\[2pt]
=-\frac{\pi}{2c(x)} \langle e^{-c(x)|t|}\ast \left(c(x)-\partial_t\right)(\frac{c(x)^{2}}{4}+\partial_{x}^{2}) \delta(t-\omega x),\varphi \rangle_{S',S} \nonumber\hspace{0.6cm}\\[2pt] =\left\langle \kappa_{{ Q}_2}(x,t),\varphi\right\rangle_{S',S}.\hspace{6.8cm}
\end{eqnarray}
Taking
\begin{equation}
\kappa_{\tilde{F}}(x,t)=(\frac{c(x)^{2}}{4}+\partial_{x}^{2})\kappa_F(x,t),   \quad \kappa_F(x,t)=\delta(t-\omega \cdot x),
\end{equation}
we get 
\begin{eqnarray}
\kappa_Q(x,t)=((1-{\bf n} \cdot \omega)+\frac{c(x)}{2})\kappa_F(x,t)-\frac{\pi}{2c(x)}  e^{-c(x)|t|}\ast \left(c(x)-\partial_t\right)\kappa_{\tilde{F}}(x,t)\nonumber \\=\kappa_{Q_1}(x,t)+\kappa_{Q_2}(x,t), \hspace{6.7cm}
\end{eqnarray}
where $\kappa_{{ Q}_1}(x,t)$  indicates the kernel of  $Q_1$. 
This allows us to write
\begin{eqnarray}
Q=Q_{1}- \frac{\pi}{2c(x)}T_{e^{-c(x)|t|}}(c(x)-\partial_t)\tilde{F},
\end{eqnarray}
 where $T_{e^{-c(x)|t|}}$ is the convolution operator of $e^{-c(x)|t|}$ and $\tilde{F}=(\frac{c^2(x)}{4}+\partial_{x}^{2})F$. 
\end{proof}

The next theorem is concerned with the computation of the amplitude of the operator \eqref{Qapp2}.

\begin{theorem}
Let  $K$ and $J$ be elliptic Fourier  integral operators of order 0. 
Then the operator $Q$ \eqref{Qapp2} and its kernel $\kappa_Q$ can respectively be written as
\begin{eqnarray}
Q=Q_{1}- \frac{\pi}{2c(x)}T_{e^{-c(x)|t|}}(c(x)J\mathbb{A}^{-1} P^{\#}_1K-J \mathbb{A}^{-1}P^{\#}_2K),
\end{eqnarray}
\begin{equation}
 \kappa_ Q(x,t)=\kappa_{{ Q}_1}(x,t)+ \kappa_ {{ Q}_2}(x,t),
\end{equation}
where $\kappa_{{ Q}_1}(x,t)$ is the kernel of $Q_1$, and $\kappa_{{ Q}_2}$ is defined by 
\begin{eqnarray}
\kappa_{{Q}_2}(x,t)=-\frac{\pi}{2}\int e^{-c(x)|t-r|+i\psi_1(x,\xi,k)-ikr}[b^{\#}(x,\xi,k)-\frac{1}{c(x)} a^{\#}(x,\xi,k)]\frac{1}{A_+}(k^{-1/3}\xi_1) d\xi dk dr. \nonumber \hspace{-2cm} 
\end{eqnarray}
Here $P^{\#}_1\in \Psi^{-n/2-1/6}$, $P^{\#}_2\in \Psi^{-n/2+5/6}$, $a^{\#}(x,\xi,k)=(\frac{c^2(x)}{4}+\partial_{x}^{2})a(x,\xi,k)$, and $b^{\#}(x,\xi,k)=(\frac{c^2(x)}{4}+\partial_{x}^{2})b(x,\xi,k)$ where  $a(x,\xi,k)$ and $b(x,\xi,k)$ are defined in Theorem \ref{th4}. Furthermore, the amplitude 
 $a_Q$ is given by 
\begin{equation}\label{aq22}
 a_ Q(x,k)=a_{Q_1}(x,k)+ a_ {Q_2}(x,k),
\end{equation}
where $a_{Q_1}(x,k)$ is the amplitude \eqref{Rayhnahbel} and 
 \begin{eqnarray}
a_{Q_2}(x,k)=-\frac{1}{2}\frac{c(x)}{c^2(x)+k^2}\int e^{ik\psi_{2}(x,\zeta)}[b^{\#}_1(x,\zeta,k)-\frac{1}{c(x)}a^{\#}_1(x,\zeta,k)]\frac{1}{A_{+}}(k^{2/3}\zeta_{1}) d\zeta, \nonumber  
\end{eqnarray}
such that $a^{\#}_1(x,\zeta,k)=k^na(x,k\zeta,k) \in S^{n/2+7/6}$, $b^{\#}_1(x,\zeta,k)=k^nb(x,k\zeta,k) \in S^{n/2+1/6}$, and  $\psi_2(x,\zeta)$ is given in Theorem \ref{th4}.  
\end{theorem}
\begin{proof}
Using (\ref{F}) and (\ref{Qapp2}) we obtain 
\begin{eqnarray}
Q=Q_{1}- \frac{\pi}{2c(x)}T_{e^{-c(x)|t|}}(c(x)J(\tilde{E}_1 \mathbb{A}_i+ \tilde{E}_2\mathbb{A}_i')K-J(\tilde{E}_3 \mathbb{A}_i+ \tilde{E}_4\mathbb{A}_i')K),
\end{eqnarray}
where $\tilde{E}_1 \in \Psi ^{-n/2+1/6}$,  $\tilde{E}_2 \in \Psi ^{-n/2-1/6}$, $\tilde{E}_3 \in \Psi ^{-n/2+1/6+1}$, and  $\tilde{E}_4 \in \Psi ^{-n/2-1/6+1}$. Using Theorem 6.5 in \cite{tay1}, we get  
\begin{eqnarray}
Q=Q_{1}- \frac{\pi}{2c(x)}T_{e^{-c(x)|t|}}(c(x)J\mathbb{A}^{-1} P^{\#}_1K-J \mathbb{A}^{-1}P^{\#}_2K),
\end{eqnarray}
such that $P^{\#}_1\in \Psi^{-n/2-1/6}$ with symbol 
\begin{equation}
p^{\#}_1(x,\xi,k)=(\frac{c^2(x)}{4}+\partial_{x}^{2})p_2(x,\xi,k), 
\end{equation}
and $P^{\#}_2 \in \Psi^{-n/2+5/6}$
 with symbol 
\begin{equation}
p^{\#}_2(x,\xi,k)=(\frac{c^2(x)}{4}+\partial_{x}^{2})p_1(x,\xi,k),
\end{equation}
where $p_1(x,\xi,k)$ and $p_2(x,\xi,k)$ are described in Theorem \ref{th4}. 
 As in the first order case, we use now the Dirac delta function $\delta(x,t)\in \mathcal{E}^{\prime }(B\times \mathbb{R})$ to find the kernel of $Q$ at the base point. Let $\varphi (x,t)\in S(\mathbb{R}^n \times \mathbb{R}) $, thus we have 
\begin{eqnarray}
\langle Q\delta _{0},\varphi\rangle_{S',S}=\langle Q_1\delta _{0},\varphi\rangle_{S',S} - \langle \frac{\pi}{2c(x)}e^{-c(x)|t|}\ast[(c(x)JP^{\#}_3\mathbb{A}^{-1}-JP^{\#}_4\mathbb{A}^{-1}]\delta _{0},\varphi \rangle_{S',S}\nonumber  \hspace{-2.5cm}\\=\langle Q_1\delta _{0},\varphi\rangle_{S',S} +\langle Q_2\delta _{0},\varphi\rangle_{S',S}, \hspace{3.5cm}\label{Q2int}
\end{eqnarray}
with $P^{\#}_3=P^{\#}_1K$ and $P^{\#}_4=P^{\#}_2K$. The term $\langle Q_1\delta _{0},\varphi\rangle_{S',S}$ is already computed \eqref{Q}. Therefore, we only need  
 \begin{eqnarray}
 \langle Q_2\delta _{0},\varphi\rangle_{S',S}=-\langle \frac{\pi}{2c(x)}e^{-c(x)|t|}\ast A\delta _{0},\varphi\rangle_{S',S}\nonumber \\=-\frac{\pi}{2c(x)}\int e^{-c(x)|t-r|}\langle A\delta _{0} ,\varphi\rangle_{S',S} dr, 
 \hspace{-1cm} 
\end{eqnarray}
with
 \begin{equation}\label{AAAA}
  A=c(x)JP^{\#}_3\mathbb{A}^{-1}-JP^{\#}_4\mathbb{A}^{-1}.
 \end{equation}
On the other hand, we can write 
\begin{eqnarray}
\langle A\delta _{0} ,\varphi\rangle_{S',S}=\int e^{i\psi_1(x,\xi,k)-ikt}[c(x)b^{\#}(x,\xi,k)-a^{\#}(x,\xi,k)]\frac{1}{A_+} (k^{-1/3}\xi_1) d\xi dk, 
\end{eqnarray}
where $b^{\#}(x,\xi,k)=p^{\#}_3(x,\xi,k)a_J(x,\xi,k)$ and $a^{\#}(x,\xi,k)= p^{\#}_4(x,\xi,k)a_J(x,\xi,k)$. 
 This leads to 
 \begin{eqnarray*}
  \langle Q_2\delta _{0},\varphi\rangle_{S',S}=-\frac{\pi}{2c(x)}\int e^{-c(x)|t-r|+i\psi_1(x,\xi,k)-ikr}[c(x)b^{\#}(x,\xi,k)-a^{\#}(x,\xi,k)]\frac{1}{A_+} (k^{-1/3}\xi_1) d\xi dk dr. 
 \end{eqnarray*}
 Finally, we find
\begin{eqnarray}
Q\delta _{0} =Q_1\delta _{0}-\frac{\pi}{2}\int e^{-c(x)|t-r|+i\psi_1(x,\xi,k)-ikr}[b^{\#}(x,\xi,k)-\frac{1}{c(x)} a^{\#}(x,\xi,k)]\frac{1}{A_+}(k^{-1/3}\xi_1)d\xi dk dr. \nonumber \hspace{-2cm} 
\label{Q_2}
\end{eqnarray}
This shows that the kernel $\kappa_Q$ is as follows 
 \begin{equation}
 \kappa_{Q}(x,t)=\kappa_{Q_1}(x,t)+\kappa_{Q_2}(x,t),
\end{equation}  
   and  
   \begin{eqnarray}
\kappa_{Q_2}(x,t)=-\frac{\pi}{2} \int e^{-c(x)|t-r|+i\psi_1(x,\xi,k)-ikr}[b^{\#}(x,\xi,k)-\frac{1}{c(x)} a^{\#}(x,\xi,k)]\frac{1}{A_+}(k^{-1/3}\xi_1)d\xi dk dr,  \nonumber \hspace{-2cm} 
\end{eqnarray}
with $a^{\#}(x,\xi,k)=(\frac{c^2(x)}{4}+\partial_{x}^{2})a(x,\xi,k)$ and $b^{\#}(x,\xi,k)=(\frac{c^2(x)}{4}+\partial_{x}^{2})b(x,\xi,k)$.   
To obtain the amplitude $a_Q$, we take  the inverse Fourier transform of $\kappa_Q$. The one related to  $\kappa_{Q_1}$ is given by \eqref{Rayhnahbel}. First, we can write that 
\begin{equation}
\kappa_{Q_2}(x,t)=-\frac{\pi}{2}e^{-c(x)|t|}\ast \kappa_A(x,t),
\end{equation} 
where $\kappa_A(x,t)$ is the kernel of the operator \eqref{AAAA}. 
Using the inverse Fourier transform, we find
\begin{equation}
\mathcal{F}^{-1}(\kappa_{Q_2})(x,k)=-\frac{\pi}{2}
\mathcal{F}^{-1}(e^{-c(x)|t|})(k)\mathcal{F}^{-1}(\kappa_A)(x,k).
\end{equation} 
Knowing that   $\mathcal{F}^{-1}(e^{-c(x)|t|})={\frac{1}{\pi}}\frac{c(x)}{c^2(x)+k^2}$, we obtain
\begin{eqnarray}
a_{Q_2}(x,k)=-\frac{1}{2}\frac{c(x)}{c^2(x)+k^2}\int e^{i\psi_{1}(x,\xi,k)}[b^{\#}(x,\xi,k)-\frac{1}{c(x)}a^{\#}(x,\xi,k)]\frac{1}{A_{+}}(k^{-1/3}\xi_{1})d\xi. \nonumber
\end{eqnarray}
Applying the change of variable $\xi=k\zeta$ with $\xi \in \mathbb{R}^n$, we get 
   \begin{eqnarray}
a_{Q_2}(x,k)=-\frac{1}{2}\frac{c(x)}{c^2(x)+k^2}\int e^{ik\psi_{2}(x,\zeta)}[b^{\#}_1(x,\zeta,k)-\frac{1}{c(x)}a^{\#}_1(x,\zeta,k)]\frac{1}{A_{+}}(k^{2/3}\zeta_{1})d\zeta,   \label{aQ222}
\end{eqnarray}
such that $a^{\#}_1(x,\zeta,k)=k^na^{\#}(x,k\zeta,k)$, $b^{\#}_1(x,\zeta,k)=k^nb^{\#}(x,k\zeta,k)$, and $\psi_2(x,\zeta)$ is described in     Theorem \ref{th4}. 
\end{proof}

The next theorem gives the asymptotic expansion  of $\eqref{aq22}$. 

\begin{theorem}
The asymptotic expression of the amplitude \eqref{aq22} is given by
\begin{eqnarray}
\nonumber a_Q(x,k)= a_{Q_{1}}(x,k)\hspace{12.0cm} \\[4pt]
-\frac{1}{2}\frac{c(x)}{c^2(x)+k^2}\sum_{p,l=0}^{P,L}k^{2/3-p-2l/3}\left[b^{\#}_{p,l}(\omega,x)- \frac{1}{c(x)}a^{\#}_{p,l}(\omega,x)\right]\psi^{(l)}(k^{1/3}Z)e^{ikx\omega}  +R_{P,L}(k), \label{2020}
\end{eqnarray}
where 
\begin{equation}\label{summm}
a_{Q_{1}}(x,k) = \sum_{p,l=0}^{P,L} k^{2/3-p-2l/3}\left((1-{\bf n} \cdot \omega)a_{p,l}(\omega,x)+ \mathcal{C}(x) b_{p,l}(\omega,x)\right)\psi^{(l)}(k^{1/3}Z). 
\end{equation}
Here  $a^{\#}_{p,l}(\omega,x)=(\frac{c^2(x)}{4}+\partial_{x}^{2})a_{p,l}(\omega,x)$, $b^{\#}_{p,l}(\omega,x)=(\frac{c^2(x)}{4}+\partial_{x}^{2})b_{p,l}(\omega,x)$, $p \in \{0,1..,P\}$, $l \in \{0,1..,L\}$,  
   $\omega$ is the incidence direction, and $c(x)>0$ is the curvature. In addition, $Z(\omega,x)$ is a continuous real function  that is positive on the illuminated region, negative on the shadow region, and vanishing on the shadow boundary. The functions $a_{p,l}$, $b_{p,l}$, $a^{\#}_{p,l}$, and $b^{\#}_{p,l}$ result from the expansion of the symbol  and the application of the stationary phase method.  
\end{theorem}

\begin{proof}
The derivation of \eqref{2020} is based on the application of the stationary phase method to the   amplitude $a_Q(x,k)=a_{Q_1}(x,k)+a_{Q_1}(x,k)$
where 
\begin{eqnarray}\label{kraht}
a_{Q_1}(x,k)=k^{1/3}\int e^{ik\psi_2(x,\zeta)+ikt\zeta_1 +ik \frac{t^3}{3}}\left[(1-{\bf n} \cdot \omega)a_{1}(x,\zeta ,k)+\mathcal{C}(x)b_{1}(x,\zeta ,k )\right] \Psi(k^{1/3}t)dt d\zeta,
\end{eqnarray}
see \eqref{aQ2}, and
\begin{eqnarray}
a_{Q_2}(x,k)=-\frac{1}{2}\frac{k^{1/3}c(x)}{c^2(x)+k^2}\int e^{ik\psi_{2}(x,\zeta)+ikt\zeta_1+ikt^3/3}[b^{\#}_1(x,\zeta,k)-\frac{1}{c(x)}a^{\#}_1(x,\zeta,k)]\Psi(k^{1/3}t) d\zeta dt, 
\end{eqnarray}
obtained using \eqref{Assam} in \eqref{aQ222}. The critical points are the same as the ones given in the paper \cite{tay1} and used when applying 
     the stationary phase method to \eqref{kraht} to derive  \eqref{sum.aQ}.
      Therefore, the latter method for the amplitude $a_Q$ leads to
      the asymptotic expression 
 \begin{eqnarray}
 \nonumber
a_Q(x,k) = a_{Q_{1}}(x,k)-\frac{1}{2}\frac{c(x)}{c^2(x)+k^2}\sum_{p,l=0}^{P,L}k^{2/3-p-2l/3}\left[b^{\#}_{p,l}(\omega,x)- \frac{1}{c(x)}a^{\#}_{p,l}(\omega,x)\right]\psi^{(l)}(k^{1/3}Z(\omega,x))e^{ikx \cdot \omega} \\
              +R_{P,L}(k),\nonumber 
\end{eqnarray}
such that $a_{Q_{1}}(x,k)$ is \eqref{summm} (first order expansion \eqref{sum.aQ}),  $a^{\#}_{p,l}(\omega,x)=(\frac{c^2(x)}{4}+\partial_{x}^{2})a_{p,l}(\omega,x)$, and $b^{\#}_{p,l}(\omega,x)=(\frac{c^2(x)}{4}+\partial_{x}^{2})b_{p,l}(\omega,x)$. The remainder $R_{P,L}$ satisfies 
\begin{equation}
|R_{P,L}(k)|\leq C_{PL}k^{-min(2L/3,P+1/3)},
\end{equation} 
and $C_{PL}$ is a constant depending on $P$ and $L$. 

\end{proof}

\begin{remark}
As is mentioned in Remark \ref{rem1}, the second order condition \eqref{appr02} is derived in two dimensions. 
If the three dimensional absorbing boundary condition 
\begin{equation}
{\partial_{\bf n} w^s(x,k)}=-(ik-c(x))w^i(x,k)-\frac{c^2(x)(c(x)+ik)}{2(c(x)^2+k^2)}\partial^2_x w^i(x,k),\label{condition3}
\end{equation}
is used (see condition (29) in\cite{Turkel1}),  then we obtain 
\begin{eqnarray}
a_Q(x,k) \sim a_{Q_{1}}(x,k)-\frac{1}{2}\frac{c^2(x)}{c^2(x)+k^2}\sum_{p,l=0}^{P,L}k^{2/3-p-2l/3}\left[ c(x)\widetilde{b}_{p,l}(\omega,x)- \widetilde{a}_{p,l}(\omega,x)\right]\psi^{(l)}(k^{1/3}Z(\omega,x)),  
\end{eqnarray}
where $a_{Q_{1}}(x,k)$ is given by \eqref{summm} with $\mathcal{C}(x)=c(x)$,  
 $\widetilde{a}_{p,l}(\omega,x)=\partial_{x}^{2}a_{p,l}(\omega,x)$, and $ \widetilde{b}_{p,l}(\omega,x)=\partial_{x}^{2}b_{p,l}(\omega,x)$.
\end{remark}

\section{Conclusion}

  In this paper we have derived some new expansions of the normal derivative of the total field solution of the Helmholtz equation.  The original expansions are based on a pseudo-differential decomposition of the Dirichlet to Neumann operator. In this work, we have used approximations of this operator to derive new expansions.
One of the goals is to facilitate construction  
of a new  ansatz class that can be used in the development of numerical solvers that can produce  more accurate solutions.

\section{Acknowledgments}

Y. Boubendir gratefully acknowledges support from NSF through grant No. DMS-1319720.


\begin{thebibliography}{99}

\bibitem {Colton-Kress:83}D. Colton and R. Kress. \textit{Integral Equation
Methods in Scattering Theory}, 1987. John Wiley and Sons, New
York.

\bibitem{Chandler-WildeEtAl12}
Chandler-Wilde, S.N., Graham, I.G., Langdon, S., Spence, E.A.
\textit{Numerical-asymptotic boundary integral methods in high-frequency acoustic scattering}, 
Acta Numerica \textbf {21} (2012), 89--305.

\bibitem{Amini1} Amini, S, and P. Harris, \emph{A comparison between various boundary integral formulations of the exterior acoustic problem}, Computer Meth. Appl. Mech. Eng. 84, 1990, 59-75.

\bibitem{Burton} Burton, A., Numerical solution of acoustic radiation problems, \emph{NPL Contract Rept. OC5/S35 National Physical Laboratory, Teddington, Middlesex}, (1976).


\bibitem {BrackhageWerner} Brackhage, H., and P. Werner, \emph{Uber das Dirichletsche
Aussenraumproblem fur die Helmholtsche Schwingungsgleichung},
Arch.Math, 16, 325-329, 1965.

\bibitem {Antoine} Antoine X. and Darbas M., \emph{Generalized combined
  field integral equations for the iterative solution of the
  three-dimensional Helmholtz equation}, Mathematical
    modeling and numerical analysis, 41, 2007, pp. 147--167.


\bibitem{AntoineX} Antoine, X., Darbas, M., \emph{Alternative integral equations for the iterative solution of acoustic scattering problems}, Q. Jl. Mech. Appl. Math. 58 (2005), no. 1, 107-128


\bibitem{Levadoux} Levadoux, D., \textit{Etude d'une \'equation int\'egrale adaptae\'e \`a la resolution haute-fr\'equance de l'equation D'Helmholtz}, \emph{Th\`ese de doctorat de l'Universit\'e de Paris VI France}, 2001.


\bibitem {Antoine-Bendali-Darbas:05}X. Antoine, A. Bendali and M. Darbas.
\textit{Analytic Preconditioners for the Boundary Integral
Solution of the Scattering of Acoustic Waves by Open Surfaces},
2005, Journal of Computational Acoustics, volume 13(3), pages
477--498.



\bibitem{Panich} Panich, I., \emph{On the question of the solvability of the exterior boundary problem for the wave equation and Maxwell's equation}, Uspekhi Mat. Nauk 20, (1965), 221-226.


\bibitem{br-turc} Bruno, O., Elling, T., Turc, C., \emph{Regularized integral equations and fast high-order solvers for sound-hard acoustic scattering problems}, submitted to Int. J. Eng. Math., 2011.


\bibitem{turc9} Anand, A., Ovall, J., Turc, C., \emph{Well-conditioned boundary integral equations for two-dimensional sound-hard scattering problems in domains with corners}, to appear in Journal Integral Equations and Applications, 2011.


\bibitem{Grenn1} L. Greengard, H. Cheng, V. Rokhlin,
\emph{A Fast Adaptive Multipole Algorithm in Three Dimensions}.
 J. Comput. Phys. 155, 468 (1999)


\bibitem {Rokhlin:90}V. Rokhlin. \textit{Rapid solution of integral equations
of scattering theory in two dimensions,} 1990; J. Comput. Phys,
6(2):414-439.


\bibitem{TongChew10}
Tong, M.S., Chew, W.C.:
Multilevel fast multipole acceleration in the Nystr\"om discretization of surface electromagnetic integral equations for composite objects,
IEEE Trans. Antennas and Propagation \textbf{58} (2010), no. 10, 3411--3416.

\bibitem{1b}  Y. Boubendir, V. Dominguez, D. Levadoux, and C. Turc.  \textit{Regularized combined field
integral equations for acoustic transmission problems}.  SIAM J. Appl. Math., 75(3), 929?952, 2015. 


\bibitem{fatihnew} F. Ecevit and H. H. Eruslu. \textit{Efficient Galerkin schemes for high-frequency scattering problems
based on frequency dependent changes of variables}. ArXiv e-prints, September 2016. 

\bibitem{Yimaturc}Y. Boubendir, C. Turc. \textit{Wavenumbers estimates for regularized combined field boundary
integral operators in acoustic scattering problems with Neumann boundary conditions,}
IMA Numerical Analysis, Volume 33 Issue 4 October 2013.

\bibitem{5b}  Y. Boubendir, and C. Turc. \textit{Well-conditioned boundary integral formulations for the
solution of high-frequency electromagnetic scattering problems}. CAMWA, V. 67, No. 10 Pg. 1772-1805, 2014. 

\bibitem{Hormander}  L. H$\ddot{\text{o}}$rmander, 
\textit{Fourier integral operators}, I, Acta Math, 127, 79-183, 1971. 

\bibitem{oscar} O. Bruno, C. Geuzaine,  F. Reitich,  
\textit{On the O(1)solution of multiple-scattering problems}.  IEEETrans. Magn. 41 (2005), 1488-1491.

\bibitem{Hormander1} H$\ddot{\text{o}}$rmander, L.  \textit{Pseudo-differential operators and hypoelliptic equations. Singular integrals}. (Proc. Sympos. Pure Math., Vol. X, Chicago, Ill., 1966), 138–183; Amer. Math. Soc., Providence, R.I. (1967).

\bibitem{victor}
Dominguez, V., Graham, I., and V. Smyshlyaev, \textit{ A hybrid numerical-asymptotic boundary integral method for 
high-frequency acoustic scattering}. Numerische Mathematik, 106,
 471-510, 2007.


\bibitem{fatih1} F. Ecevit and F. Reitich,
\emph{Analysis of multiple scattering iterations for
high-frequency scattering problems. I: The two-dimensional case}.
 Numerische Mathematik 114:271-354, 2009. 

\bibitem{fatih2}A. Anand, Y. Boubendir, F. Ecevit and F. Reitich,
\textit{Analysis of multiple scattering iterations for
high-frequency scattering problems. II: The three-dimensional
scalar case}. Numerische Mathematik, 114:373-427, 2010.


\bibitem{Fernando1}O. P. Bruno, C. Geuzaine, J-A. Monro and F. Reitich,
\emph{Prescribed error tolerance within fixed computational times
for scattering problems of arbitrarily high frequency: the convex
case}. Phil. Trans. Roy. Soc. London 362 (2004), 629-645.

%


\bibitem{Turkel1} M. Medvinsky, E. Turkel, U. Hetmaniuk. \textit{ 
Local absorbing boundary conditions for elliptical shaped boundaries}. Journal of Computational Physics 227,
 8254-8267, 2008.

\bibitem{HuybrechsVandewalle07}
Huybrechs, D., Vandewalle, S.:
A sparse discretization for integral equation formulations of high frequency scattering problems,
SIAM J. Sci. Comput. \textbf{29} (2007), no. 6, pp. 2305--2328.


\bibitem{Giladi07}
Giladi, E.:
An asymptotically derived boundary element method for the Helmholtz equation in high frequencies,
J. Comput. Appl. Math. \textbf{198} (2007), 52--74.

\bibitem{AbboudEtAl94} Abboud, T., N\'ed\'elec, J.-C., Zhou, B.
\textit{M\'ethode des \'equations int\'egrales pour les hautes fr\'equences }, 
C.R. Acad. Sci. Paris \textbf{318} (1994), 165--170.

\bibitem{AbboudEtAl95} Abboud, T., N\'ed\'elec, J.-C., Zhou, B. 
\textit{Improvement of the integral equation method for high frequency problems},
in Mathematical and Numerical Aspects of Wave Propagation: Mandelieu-La Napoule, SIAM,
(1995), 178--187.


\bibitem{tay1} R. Melrose,  B. Taylor, \textit{Near peak scattering and the corrected Kirchhoff approximation for a convex obstacle}.
Adv. in Math. 55 (1985), 242-315.


\bibitem{EcevitOzen14} Ecevit, F. \& Ozen, H.C. \textit{Frequency-adapted galerkin boundary element methods for convex scattering problems}. Numer. Math. {doi}:10.1007/s00211-016-0800-7, 2016.


\bibitem{Bayliss} A. Bayliss, M. Gunzburger, and E. Turkel. \textit{boundary conditions for the numerical solutions for elliptic equations in exterior regions}. SIAM J. Appl. Math, 42 :430–451, 1982.

\bibitem{MonoMelrose} R. Melrose and M. Taylor, 
\textit{Boundary Problems for Wave Equations With Grazing and Gliding Rays}. Monograph, in preparation. 

\bibitem{ClaudeZuily02}

Claude Zuily .:El\'{e}ments de distributions et d'\'{e}quations aux d\'{e}riv\`{e}es partielles. - Cours et probl\'{e}mes r\'{e}solus, Dunod. 2002. 


\bibitem{AndreMartinez}

Andr\'{e} Martinez,  \textit{
An Introduction to Semiclassical and Microlocal Analysis}, springer, 2000.


\bibitem{BrunoSayas} O.P. Bruno, V. Domínguez, F.J. Sayas,
\textit{Convergence analysis of a high-order Nyström integral-equation method for surface scattering problems}, Numer. Math. (2013) 124: 603. 


\bibitem{Taylorbook}M. Taylor. \textit{ Pseudodifferential Operators}. Princeton University Press, Princeton, 1981.

\bibitem{Wilcox} C. H. Wilcox. \textit{Scattering theory for the d’Alembert equation in exterior domains}, volume 442. Springer-Verlag, Berlin 1975.


\bibitem{Chazarain}J. Chazarain and A. Piriou. \textit{Introduction to the theory of linear partial differential equations}. North-Holland, 1982.


\end{thebibliography}
\end{document}